\documentclass[10pt, reqno]{amsart}
\usepackage{amsmath, amsthm, amscd, amsfonts, amssymb, graphicx, color}
\usepackage[bookmarksnumbered, colorlinks, plainpages]{hyperref}
\hypersetup{colorlinks=true,linkcolor=red, anchorcolor=green, citecolor=cyan, urlcolor=red, filecolor=magenta, pdftoolbar=true}
\usepackage{mathrsfs}
\usepackage[OT2,T1]{fontenc}
\DeclareSymbolFont{cyrletters}{OT2}{wncyr}{m}{n}
\DeclareMathSymbol{\Sha}{\mathalpha}{cyrletters}{"58}

\newtheorem{theorem}{Theorem}[section]
\newtheorem{lemma}[theorem]{Lemma}

\newtheorem{corollary}[theorem]{Corollary}
\theoremstyle{definition}

\theoremstyle{remark}
\newtheorem{remark}[theorem]{Remark}
\numberwithin{equation}{section}

\begin{document}
\setcounter{page}{1}

\title[Quantum dynamics of elliptic curves]{Quantum dynamics of elliptic curves}

\author[Nikolaev]
{Igor V. Nikolaev$^1$}

\address{$^{1}$ Department of Mathematics and Computer Science, St.~John's University, 8000 Utopia Parkway,  
New York,  NY 11439, United States.}
\email{\textcolor[rgb]{0.00,0.00,0.84}{igor.v.nikolaev@gmail.com}}

\dedicatory{All data are part of the manuscript}



%\date{Received:  August 14, 2015; Revised: yyyyyy; Accepted: zzzzzz.}

\begin{abstract}
We calculate $K$-theory   of a  crossed product $C^*$-algebra 
 of  the  noncommutative torus with real multiplication  by  elliptic curve  $\mathscr{E}(K)$
 over a number field $K$.    This result is used  to evaluate  the rank and  
 the Shafarevich-Tate group of $\mathscr{E}(K)$. 
 \end{abstract}

\maketitle

%**************************************************************************
\section{Introduction}
%***************************************************************************
Let  $\mathscr{A}_{\theta}$  be a  noncommutative torus, i.e.  
the $C^*$-algebra on the  generators
$u$ and $v$ satisfying the relation $vu=e^{2\pi i\theta}uv$ for a real constant $\theta$
[Rieffel 1990] \cite{Rie1}. Recall that  $\mathscr{A}_{\theta}$ is said to have real multiplication (RM)  if  $\theta$ is an
irrational quadratic  number  [Manin 2004] \cite{Man1};  we  denote such a $C^*$-algebra by $\mathscr{A}_{RM}$.
Let $K$ be a number field and let $\mathscr{E}(K)$ be an elliptic curve over $K$.  
Here we  consider   a  functor $F$ between elliptic  curves $\mathscr{E}(K)$ and 
 the $C^*$-algebras $\mathscr{A}_{RM}$  \cite[Section 1.3]{N}. Such a functor 
 maps $K$-isomorphic elliptic curves $\mathscr{E}(K)$  to 
  isomorphic $C^*$-algebras $\mathscr{A}_{RM}$.

Recall that $\mathscr{E}(K)$ is  an algebraic group over $K$.
The Mordell-Weil Theorem says that 
$\mathscr{E}(K)\cong \mathbf{Z}^r\oplus \mathscr{E}_{tors}(K)$, 
where $r=rk~\mathscr{E}(K)\ge 0$ is the rank of $\mathscr{E}(K)$ and $\mathscr{E}_{tors}(K)$
is a finite abelian group. 
The group operation $\mathscr{E}(K)\times \mathscr{E}(K)\to \mathscr{E}(K)$ 
defines an action of  the group  $\mathscr{E}(K)$  by the $K$-automorphisms
of $\mathscr{E}(K)$.  Likewise each $K$-automorphism of $\mathscr{E}(K)$   gives rise  
to an automorphism of  $\mathscr{A}_{RM}$.  
Thus  one gets an action of the group  $\mathscr{E}(K)$  by  automorphisms 
of the $C^*$-algebra $\mathscr{A}_{RM}$.  
We study   a crossed product
 corresponding to  this action,  i.e. the  $C^*$-algebra:
%************************************************************************************
\begin{equation}\label{eq1.1}
\mathscr{A}_{RM}\rtimes \mathscr{E}(K).
\end{equation}
%***********************************************************************************
The crossed product (\ref{eq1.1})   is an example  the quantum dynamical system
[Blackadar 1986] \cite[Chapter V]{B} and  [Pedersen 1979] \cite{P};   hence the title. 
The aim of our note is  the $K$-theory of crossed product 
 (\ref{eq1.1}) (Theorem \ref{thm1.1}). 
 We apply Theorem \ref{thm1.1}  to  the  rank and  the Shafarevich-Tate group of 
  $\mathscr{E}(K)$ (Corollary \ref{cor1.2}). 
  We shall use the following notation.

  Recall that $K_0(\mathscr{A}_{\theta})\cong\mathbf{Z}^2$ [Rieffel 1990] \cite{Rie1}.
 Let  $\tau$  be a  tracial state on the  crossed product  $\mathscr{A}_{RM}\rtimes \mathscr{E}(K)$ 
 [Phillips 2005] \cite[Theorem 3.4]{Phi1}.
 It is known that  $\tau$ defines an embedding
 $K_0(\mathscr{A}_{\theta})\hookrightarrow \mathbf{R}$ given by the formula
 $\tau(K_0(\mathscr{A}_{\theta}))=\mathbf{Z}+\theta\mathbf{Z}\subset\mathbf{R}$  [Blackadar 1986]  \cite[Exercise 10.11.6]{B};
 the latter will be called a  pseudo-lattice [Manin 2004] \cite{Man1}.  
 Let  $\Lambda$ be a ring of endomorphisms of  the  pseudo-lattice  
$\tau(K_0(\mathscr{A}_{RM}))$. Such a ring is an order  $\Lambda\cong \mathbf{Z}+fO_k$ in the 
field $k=\mathbf{Q}(\theta)$, where $O_k$ is the ring of integers  of $k$ and $f\ge 1$ is a conductor of the order.  
We  denote by  $Cl~(\Lambda)$  the class group of the ring $\Lambda$
and by  $h_{\Lambda}=|Cl~(\Lambda)|$  the class number of 
$\Lambda$.  Moreover, let  $\mathcal{K}_{ab}$ 
the maximal abelian extension of the field $k$ modulo conductor $f$.  
(If $f=1$, then the extension  $\mathcal{K}_{ab}$ is  unramified, i.e. the Hilbert
class field of $k$.)   It is known,  that $Gal~(\mathcal{K}_{ab}|k)\cong  Cl~(\Lambda)$,
where   $Gal~(\mathcal{K}_{ab}|k)$ is the Galois group of the extension $k\subseteq  \mathcal{K}_{ab}$. 
 Let $\{\alpha_i ~|~1\le i\le h_{\Lambda}\}$ be generators 
of  the field $\mathcal{K}_{ab}$,  such that  $\alpha_i$ are conjugate algebraic numbers.
Consider a  normalization of $\alpha_i$ given by the formula
$\{\lambda_i=\alpha_i \alpha^{-1}_{h_{\Lambda}} ~|~1\le i\le h_{\Lambda}-1\}$.
 Our main results can be formulated as follows. 
 %************************************************************************************
\begin{theorem}\label{thm1.1}
The $K$-theory of  crossed product (\ref{eq1.1}) is as follows:
%*******************************************************************
\begin{equation}\label{eq1.2}
\left\{
\begin{array}{lll}
K_0(\mathscr{A}_{RM}\rtimes \mathscr{E}(K)) &\cong & \mathbf{Z}^{h_{\Lambda}+1},\\
&&\\
\tau(K_0(\mathscr{A}_{RM}\rtimes \mathscr{E}(K))) &=& \mathbf{Z}+\theta\mathbf{Z}+\lambda_1\mathbf{Z}+\dots+
\lambda_{h_{\Lambda}-1}\mathbf{Z}.
\end{array}
\right.
\end{equation}
%*****************************************************************
\end{theorem}
%************************************************************************************  

\bigskip
The rank and the Shafarevich-Tate group [Silverman 1985]  \cite[Appendix B]{S} 
of   $\mathscr{E}(K)$ will be denoted by  $rk~\mathscr{E}(K)$
 and $\Sha (\mathscr{E}(K))$, respectively.  Theorem \ref{thm1.1} implies 
 the following formulas.  
%*************************************************************************
\begin{corollary}\label{cor1.2}
%*******************************************************************
\begin{equation}\label{eq1.3+}
\left\{
\begin{array}{lll}
rk~\mathscr{E}(K) &=& h_{\Lambda}-1,\\
&&\\
\Sha (\mathscr{E}(K)) &\cong&  ~Cl~(\Lambda)\oplus Cl~(\Lambda).
\end{array}
\right.
\end{equation}
%*****************************************************************
\end{corollary}
%************************************************************
%************************************************************************************
\begin{remark}\label{rm1.4}
For  simplicity, we treat the case of elliptic curves only.  Notice that 
formulas (\ref{eq1.2}) and (\ref{eq1.3+}) are true for abelian 
varieties  over  number fields  $K$ \cite{Nik2}.
\end{remark}
%************************************************************************************  
%************************************************************************************
\begin{remark}\label{rm1.5}
It follows from  (\ref{eq1.3+}) that 
%**********************************************************************************
\begin{equation}\label{eq1.3}
|\Sha(\mathscr{E}(K))|=(1+rk~\mathscr{E}(K))^2.
\end{equation}
%*************************************************************************************
Since the group $\Sha (\mathscr{E}(K))$ is hard to compute
 [Tate 1974] \cite[p.193]{Tat1}, a verification of  (\ref{eq1.3}) seems to be out of reach.  
 One can use analytic  values of  $rk~\mathscr{E}(K)$  
and  $|\Sha (\mathscr{E}(K))|$
assuming that the Birch and Swinnerton-Dyer Conjecture is true [Swinnerton-Dyer 1967]  \cite{Swi1}. 
While many of such values satisfy (\ref{eq1.3}),  the other do  not [Cremona et al.  2017] \cite{LMFDB}.  
The author is unaware of an exact  relation between the   
analytic values  and those  described by formula  (\ref{eq1.3}).   
\end{remark}
%************************************************************************************  
The article is organized as follows.  The preliminary facts are introduced in Section 2.  The 
proof of Theorem \ref{thm1.1} and Corollary \ref{cor1.2} can be found in Section 3.

%**************************************************************************
\section{Preliminaries}
%***************************************************************************
We briefly review  elementary facts about  $C^*$-algebras, class field theory for
 quadratic fields and the Shafarevich-Tate groups.
We refer the reader to   [Pedersen 1979] \cite{P}, % [Williams 2007] \cite{W}
 [Neukirch 1999] \cite{NE} and [Silverman 1985]  \cite{S} for a detailed account. 

%**************************************************************************
\subsection{$C^*$-dynamical systems}
%***************************************************************************
\subsubsection{$C^*$-algebras}
%****************************************************************************
A  $C^*$-algebra $\mathscr{A}$ is an algebra  over $\mathbf{C}$ with a norm
$a\mapsto ||a||$ and an involution $a\mapsto a^*$ such that
it is complete with respect to the norm and $||ab||\le ||a||~ ||b||$
and $||a^*a||=||a||^2$ for all $a,b\in \mathscr{A}$.
Any commutative $C^*$-algebra is  isomorphic
to the algebra $C_0(X)$ of continuous complex-valued
functions on some locally compact Hausdorff space $X$; 
otherwise, $\mathscr{A}$ can be thought of as  a noncommutative  topological
space.

%************************************************************************************
\subsubsection{$K$-theory of $C^*$-algebras}
%************************************************************************************
For a unital $C^*$-algebra $\mathscr{A}$, let $V(\mathscr{A})$
be the union over $n$ of projections in the $n\times n$
matrix $C^*$-algebra with entries in $\mathscr{A}$;
projections $p,q\in V(\mathscr{A})$ are  equivalent if there exists a partial
isometry $u$ such that $p=u^*u$ and $q=uu^*$. The equivalence
class of projection $p$ is denoted by $[p]$;
the equivalence classes of orthogonal projections can be made to
a semigroup by putting $[p]+[q]=[p+q]$. The Grothendieck
completion of this semigroup to an abelian group is called
the  $K_0$-group of the algebra $\mathscr{A}$.
The functor $\mathscr{A}\to K_0(\mathscr{A})$ maps the category of unital
$C^*$-algebras into the category of abelian groups, so that
projections in the algebra $\mathscr{A}$ correspond to a positive
cone  $K_0^+\subset K_0(\mathscr{A})$ and the unit element $1\in \mathscr{A}$
corresponds to an order unit $u\in K_0(\mathscr{A})$.
The ordered abelian group $(K_0,K_0^+,u)$ with an order
unit  is called a dimension group;  an order-isomorphism
class of the latter we denote by $(G,G^+)$.

%************************************************************************************
\subsubsection{Crossed products}
%************************************************************************************
Let $\mathscr{A}$ be a $C^*$-algebra and $G$ a locally compact group.  
We shall consider a continuous homomorphism $\alpha$ from $G$ to the 
group $Aut~\mathscr{A}$ of $\ast$-automorphisms of $A$ endowed with the topology
of pointwise norm-convergence.  
Roughly speaking,  the idea of the crossed
product construction  is to embed $\mathscr{A}$ into a larger $C^*$-algebra 
in which the automorphism becomes the inner automorphism.  
A covariant representation of the 
triple $(\mathscr{A},G,\alpha)$ is a pair of representations $(\pi,\rho)$ of $\mathscr{A}$ and $G$ on the
same Hilbert space $\mathscr{H}$, such that
%*******************************************************************
%\displaymath
$\rho(g)\pi(a)\rho(g)^*=\pi(\alpha_g(a))$
%\enddisplaymath
%******************************************************************** 
for all $a\in \mathscr{A}$ and $g\in G$. Each covariant representation of 
$(\mathscr{A},G,\alpha)$ gives rise to a convolution algebra $C(G,\mathscr{A})$
of continuous functions from $G$ to $\mathscr{A}$;  the completion of
$C(G,\mathscr{A})$  in the norm topology is a $C^*$-algebra $\mathscr{A}\rtimes_{\alpha} G$
called a crossed product of $\mathscr{A}$ by $G$. If $\alpha$ is a single
automorphism of $\mathscr{A}$, one gets an action of $\mathbf{Z}$ on $\mathscr{A}$;  
the crossed product in this case is called simply the crossed product
of $\mathscr{A}$ by ${\alpha}$.  

%********************************************************************************
\subsubsection{AF-algebras}
%*********************************************************************************
An  AF-algebra  (Approximately Finite $C^*$-algebra) is defined to
be the  norm closure of an ascending sequence of   finite dimensional
$C^*$-algebras $M_n$,  where  $M_n$ is the $C^*$-algebra of the $n\times n$ matrices
with entries in $\mathbf{C}$. Here the index $n=(n_1,\dots,n_k)$ represents
the  semi-simple matrix algebra $M_n=M_{n_1}\oplus\dots\oplus M_{n_k}$.
The ascending sequence mentioned above  can be written as 
%***********************************************************
%\begin{equation}\label{eq2}
$
M_1\buildrel\rm\varphi_1\over\longrightarrow M_2
   \buildrel\rm\varphi_2\over\longrightarrow\dots,
$
%\end{equation}
%****************************************************
where $M_i$ are the finite dimensional $C^*$-algebras and
$\varphi_i$ the homomorphisms between such algebras.  
If $\mathbb{A}$ is an AF-algebra, then its dimension group
$(K_0(\mathbb{A}), K_0^+(\mathbb{A}), u)$ is a complete isomorphism
invariant of algebra $\mathbb{A}$.  
The order-isomorphism  class $(K_0(\mathbb{A}), K_0^+(\mathbb{A}))$
 is an invariant of the Morita equivalence of algebra 
$\mathbb{A}$,  i.e.  an isomorphism class in the category of 
finitely generated projective modules over $\mathbb{A}$.

%**************************************************************************
\subsection{Abelian extensions of  quadratic fields}
%***************************************************************************
Let $D$ be a square-free integer and let $k=\mathbf{Q}(\sqrt{D})$ be a quadratic number field,
i.e. an extension of degree two of the field of rationals.   Denote by $O_k$ the ring of
integers of  $k$ and by  $\Lambda$ an order in $O_k$, i.e. a finitely generated $\mathbf{Z}$-module which is 
a subring of the ring $O_k$
containing $1$.  The order $\Lambda$ can be written in the form 
$\Lambda=\mathbf{Z}+fO_k$, where the integer  $f\ge 1$ is a conductor of $\Lambda$. 
If $f=1$,  then $\Lambda\cong O_k$ is the maximal order.

Denote by $Cl~(\Lambda)$ the ideal class group and by $h_{\Lambda}=|Cl~(\Lambda)|$ 
the class number of the ring   $\Lambda$.  If $\Lambda\cong O_k$, then $h_{\Lambda}$
coincides with the class number $h$ of the field $k$. The integer $h\le h_{\Lambda}$ is always a divisor of 
$h_{\Lambda}$   given by the formula:
%*******************************************************************************
\begin{equation}\label{eq2.1}
h_{\Lambda}=h  ~{f \over e_f}\prod_{p|f}\left(1-\left({D\over p}\right){1\over p}\right),
\end{equation}
%*****************************************************************************
where  $e_f$ is the index of the group of units of  $\Lambda$ in the group of units of $O_k$, $p$ is a prime number and 
$\left({D\over p}\right)$ is the Legendre symbol.

Let $\mathcal{K}_{ab}$ be the maximal abelian extension of the field $k$ modulo conductor $f\ge 1$. 
The class field theory says that 
%*******************************************************************************
\begin{equation}\label{eq2.2}
Gal~(\mathcal{K}_{ab}|k)\cong Cl~(\Lambda),
\end{equation}
%*****************************************************************************
where $Gal~(\mathcal{K}_{ab}|k)$ is the Galois group of the extension $(\mathcal{K}_{ab}|k)$. 
The $\mathcal{K}_{ab}$ is  the Hilbert class field (i.e. a maximal unramified abelian extension)  
of $k$  if and only if $f=1$.

For $D<0$ an  explicit construction of generators  of the field  $\mathcal{K}_{ab}$ is  realized  
by elliptic curves with  complex multiplication,  see e.g.   [Neukirch 1999] \cite[Theorem 6.10]{NE}. 
For $D>0$ an  explicit construction of  generators  of the field  $\mathcal{K}_{ab}$ is  realized  by 
noncommutative tori  with real multiplication  \cite[Theorem 6.4.1]{N}.

%**************************************************************************
\subsection{Shafarevich-Tate group of elliptic curve}
%***************************************************************************
The Shafarevich-Tate group $\Sha (\mathscr{E}(K))$ is a measure of
failure of the Hasse principle for the elliptic curve $\mathscr{E}(K)$. 
Recall that if $\mathscr{E}(K)$ has  a $K$-rational point, then it has 
also a $K_v$-point for every completion $K_v$ of the number field $K$. 
A converse of this statement is called the Hasse principle. In general, the 
Hasse principle fails for the elliptic curve $\mathscr{E}(K)$. 

Denote by $H^1(K,\mathscr{E})$ the first Galois cohomology group 
of $\mathscr{E}(K)$ [Silverman 1985] \cite[Appendix B]{S}.  There exists
a natural homomorphism
%*******************************************************************************
\begin{equation}\label{eq2.3}
\omega: H^1(K,\mathscr{E})\rightarrow \prod_v H^1(K_v,\mathscr{E}),
\end{equation}
%*****************************************************************************
where $H^1(K_v,\mathscr{E})$ is the first Galois cohomology over the
field $K_v$.  The Shafarevich-Tate group  of 
an elliptic curve  $\mathscr{E}(K)$ is  
%*******************************************************************
\begin{equation}\label{eq2.4}
\Sha (\mathscr{E}(K)):=Ker ~\omega. 
\end{equation}
%********************************************************************
The group $\Sha (\mathscr{E}(K)$ is trivial  if and only if 
elliptic curve $\mathscr{E}(K)$ satisfies the Hasse principle.
%********************************************************************
\begin{remark}
The Shafarevich-Tate group $\Sha (A(K))$ of an abelian variety $A(K)$ over 
the number field $K$ is defined similarly and has the same properties as 
$\Sha (\mathscr{E}(K))$. 
\end{remark}
%*******************************************************************

%**************************************************************************I  
\section{Proofs}
%***************************************************************************
%**************************************************************************
\subsection{Proof of theorem \ref{thm1.1}}
%***************************************************************************
For the sake of clarity, let us outline the main ideas. 
Our proof uses the following  rigidity principle 
based on the class field theory for  the 
real quadratic fields.  Namely, the canonical embedding 
$\mathscr{A}_{RM}\hookrightarrow  \mathscr{A}_{RM}\rtimes \mathscr{E}(K)$ 
implies an inclusion $K_0(\mathscr{A}_{RM})\subseteq K_0(\mathscr{A}_{RM}\rtimes \mathscr{E}(K))$.  
Using the canonical tracial state $\tau$
%****************************************************************************************
\footnote{The existence of  $\tau$ follows from [Phillips 2005] \cite[Theorem 3.4]{Phi1}.}
%****************************************************************************************
on   $\mathscr{A}_{RM}\rtimes \mathscr{E}(K)$, 
one gets an inclusion:
%***************************************************************************************************
 \begin{equation}\label{eq3.1}
 \mathbf{Z}+\theta\mathbf{Z}\subseteq \lambda_1\mathbf{Z}+\dots+\lambda_m\mathbf{Z},
 \end{equation}
 %********************************************************************************************
where $\lambda_i$ are generators of the pseudo-lattice $\tau(K_0(\mathscr{A}_{RM}\rtimes \mathscr{E}(K)))$. 
It is easy to see, that  each $\lambda_i\in\mathbf{R}$ is an integer algebraic  number. 
But the crossed product (\ref{eq1.1}) depends solely on the algebra $\mathscr{A}_{RM}$, 
see formula (\ref{eq3.5}).  Therefore the extension (\ref{eq3.1}) satisfies  certain  rigidity condition.  
Specifically,  the arithmetic of the number field $k(\lambda_i)$ must be controlled 
by the arithmetic of the  field $k$.  It is  well known,   that this happens if and only if
$k(\lambda_i)\cong \mathcal{K}_{ab}$,  where   $\mathcal{K}_{ab}$ is the maximal abelian extension 
 of the field $k$ modulo  conductor $f\ge 1$.  Thus  $m=1+h_{\Lambda}$, where 
$h_{\Lambda}$ is the class number of the order $\Lambda\subseteq O_k$; we refer the reader to
(\ref{eq2.1}) for an explicit formula.   We pass to a detailed argument by splitting 
the proof in a series of lemmas and corollaries.

\bigskip
%***********************************************************************************
\begin{lemma}\label{lm3.1}
The real numbers $\lambda_i$ in formula (\ref{eq3.1}) are algebraic integers. 
\end{lemma}
%************************************************************************************      
\begin{proof}
Recall that the endomorphism ring $\Lambda$ of the pseudo-lattice $\mathbf{Z}+\theta\mathbf{Z}$
is an order $\mathbf{Z}+fO_k$ in the number field $k$. In particular, since $f\ne 0$ we conclude
that $\Lambda$ is a non-trivial ring, i.e. $\Lambda\not\cong\mathbf{Z}$.

On the other hand, the inclusion (\ref{eq3.1}) implies that:
%***************************************************************************************************
 \begin{equation}\label{eq3.2}
 \Lambda\subseteq End~(\lambda_1\mathbf{Z}+\dots+\lambda_m\mathbf{Z}),
 \end{equation}
 %********************************************************************************************
where $End~(\lambda_1\mathbf{Z}+\dots+\lambda_m\mathbf{Z})$ is the endomorphism ring
of the pseudo-lattice $\lambda_1\mathbf{Z}+\dots+\lambda_m\mathbf{Z}\subset\mathbf{R}$.  
We conclude from (\ref{eq3.2}) that the ring $End~(\lambda_1\mathbf{Z}+\dots+\lambda_m\mathbf{Z})$
is non-trivial, i.e.  bigger than the  ring $\mathbf{Z}$.

Recall that  the endomorphisms of  pseudo-lattice $\lambda_1\mathbf{Z}+\dots+\lambda_m\mathbf{Z}\subset\mathbf{R}$
coincide with multiplication by the real numbers. In other words, the ring  $End~(\lambda_1\mathbf{Z}+\dots+\lambda_m\mathbf{Z})$
is the coefficient ring of a $\mathbf{Z}$-module $\lambda_1\mathbf{Z}+\dots+\lambda_m\mathbf{Z}\subset\mathbf{R}$
 [Borevich \& Shafarevich 1966] \cite[p. 87]{BS}.  Up to a multiple, any such ring must be an order in 
a real number field $\mathcal{K}$. Thus  we have a field extension $\mathcal{K}~|~k$ and  the following inclusions:
%*********************************************************************************************************
\begin{equation}\label{eq3.3}
\Lambda\subseteq O_k\subseteq O_{\mathcal{K}}, 
\end{equation}
%*********************************************************************************************************
where $O_{\mathcal{K}}$ is the ring of integers of the field $\mathcal{K}$.

On the other hand, it is known that the full $\mathbf{Z}$-module  $\lambda_1\mathbf{Z}+\dots+\lambda_m\mathbf{Z}$ is contained
 in its coefficient ring  $O_{\mathcal{K}}$  [Borevich \& Shafarevich 1966] \cite[Lemma 1, p. 88]{BS}.  
In particular,  each $\lambda_i$ is an algebraic integer.  Lemma \ref{lm3.1}
is proved.
\end{proof}

%************************************************************************************************
\begin{remark}\label{rmk3.2}
It is useful to scale the RHS of inclusion (\ref{eq3.1}) dividing it by the real number $\lambda_m\ne 1$.
Such a normalization is always possible, since the embedding $\tau:  K_0(\mathscr{A}_{RM}\rtimes \mathscr{E}(K))\to
\mathbf{R}$ is defined up to a scalar multiple. Thus we can rewrite inclusion (\ref{eq3.1}) in the form:
%*********************************************************************************************************
\begin{equation}\label{eq3.4}
\tau(K_0(\mathscr{A}_{RM}\rtimes \mathscr{E}(K))) = \mathbf{Z}+\theta\mathbf{Z}+\lambda_1\mathbf{Z}+\dots+
\lambda_{m-1}\mathbf{Z}.
\end{equation}
%*********************************************************************************************************
\end{remark}
%************************************************************************************************

\medskip
%***********************************************************************************
\begin{lemma}\label{lm3.2}
The number field  $k(\lambda_1,\dots,\lambda_m)$ is the maximal abelian extension
 of  the field $k$ modulo conductor $f\ge 1$.
\end{lemma}
%************************************************************************************      
\begin{proof}
Let  $\mathscr{E}(K)$ be an elliptic curve over the number field $K$ and let  $\mathscr{A}_{RM}=F(\mathscr{E}(K))$
be the corresponding noncommutative torus with real multiplication  \cite[Section 1.3]{N}.
The functor $F$ is faithful on the category of $K$-rational elliptic curves and therefore $F$ has a correctly 
defined inverse $F^{-1}$ on its image.  Thus  $\mathscr{E}(K)=F^{-1}(\mathscr{A}_{RM})$ and one can write the crossed
product (\ref{eq1.1}) in the form:
%************************************************************************************
\begin{equation}\label{eq3.5}
\mathscr{A}_{RM}\rtimes F^{-1}(\mathscr{A}_{RM}).
\end{equation}
%***********************************************************************************

Consider an endomorphism ring, $M$, of the pseudo-lattice $\lambda_1\mathbf{Z}+
\dots +\lambda_m\mathbf{Z}\subset\mathbf{R}$.  Denote by $\mathcal{K}\cong M\otimes\mathbf{Q}$ 
a  number field, such that $M$ is an order in the ring of integers $O_{\mathcal{K}}$ of the field $\mathcal{K}$. 
In view of the inclusion (\ref{eq3.2}),
one gets an inclusion $\Lambda\subseteq M$, where $\Lambda$ is an order in the 
ring $O_k$. Since $\Lambda\subseteq O_k$ and $M\subseteq O_{\mathcal{K}}$, we get the
inclusions: 
%************************************************************************************
\begin{equation}\label{eq3.6}
O_k\subseteq O_{\mathcal{K}} \quad \hbox{and} \quad k\subseteq\mathcal K.
\end{equation}
%***********************************************************************************

On the other hand, it follows from the formula (\ref{eq3.5}) that the crossed product 
$\mathscr{A}_{RM}\rtimes \mathscr{E}(K)$ depends only on the inner structure of algebra $\mathscr{A}_{RM}$. 
The same is true for the inclusions of groups $K_0(\mathscr{A}_{RM})\subseteq K_0(\mathscr{A}_{RM}\rtimes \mathscr{E}(K))$,
the inclusion of pseudo-lattices $\tau(K_0(\mathscr{A}_{RM}))\subseteq\tau(K_0(\mathscr{A}_{RM}\rtimes \mathscr{E}(K)))\subset\mathbf{R}$
and the inclusion of  rings $End~(\tau(K_0(\mathscr{A}_{RM})))\subseteq End~(\tau(K_0(\mathscr{A}_{RM}\rtimes \mathscr{E}(K))))$. 
In particular, the last inclusion says that  arithmetic of the number field $\mathcal{K}$ in formula (\ref{eq3.6}) is 
controlled by the arithmetic of the field  $k$. In other words,  there exists  an isomorphism:
%************************************************************************************
\begin{equation}\label{eq3.7}
Gal~(\mathcal{K}|k)\cong Cl~(\Lambda),
\end{equation}
%***********************************************************************************
 where $Gal~(\mathcal{K}|k)$ is the Galois group of the extension $k\subseteq\mathcal{K}$. 
 Therefore  $\mathcal{K}$ is the maximal abelian extension of the field $k$ 
 modulo conductor $f\ge 1$, see Section 2.2.

 Let us show that  $\mathcal{K}$ is an extension of the real quadratic field $k$ by the values  $\{\lambda_i\}_{i=1}^m$.
 Indeed,  since the coefficient ring of the full $\mathbf{Z}$-module $\lambda_1\mathbf{Z}+\dots +\lambda_m\mathbf{Z}$
 is isomorphic to the order $M\subseteq O_{\mathcal{K}}$,  we conclude that $\lambda_i\in  O_{\mathcal{K}}$
  [Borevich \& Shafarevich 1966] \cite[Lemma 1, p. 88]{BS}.  Moreover, by a change of basis in the $\mathbf{Z}$-module,
  we can always arrange the generators $\{\lambda_i\}_{i=1}^m$ to be algebraically conjugate numbers of the field extension $\mathcal{K}~|~k$.
  In particular, one gets   $\mathcal{K}=k(\lambda_1,\dots,\lambda_m)$.
  Lemma \ref{lm3.2} follows.   
\end{proof}

%***********************************************************************************
\begin{corollary}\label{cor3.3}
The cardinality of the set of generators $\lambda_i$  is $m=h_{\Lambda}$. 
\end{corollary}
%************************************************************************************      
\begin{proof}
Indeed, since $\{\lambda_i\}_{i=1}^m$ are algebraically conjugate numbers of the field extension 
$\mathcal{K}~|~k$, we conclude that $m=|Gal~(\mathcal{K}|k)|$.  But $Gal~(\mathcal{K}|k)\cong Cl~(\Lambda)$
and therefore $m=|Cl~(\Lambda)|=h_{\Lambda}$. Corollary \ref{cor3.3} follows. 
\end{proof}
%************************************************************************************************
\begin{remark}\label{rmk3.5}
To prove our results, we do not need an explicit formula for the values of generators $\lambda_i$ in terms of  $\theta\in k$;  
however,   we refer an  interested reader to \cite[Theorem 6.4.1]{N} for such a formula.
\end{remark}
%************************************************************************************************
%***********************************************************************************
\begin{corollary}\label{cor3.6}
$\tau(K_0(\mathscr{A}_{RM}\rtimes \mathscr{E}(K))) = \mathbf{Z}+\theta\mathbf{Z}+\lambda_1\mathbf{Z}+\dots+
\lambda_{h_{\Lambda}-1}\mathbf{Z}$. 
\end{corollary}
%************************************************************************************      
\begin{proof}
The formula follows from remark \ref{rmk3.2} and corollary \ref{cor3.3}. 
\end{proof}

%***********************************************************************************
\begin{corollary}\label{cor3.7}
$K_0(\mathscr{A}_{RM}\rtimes \mathscr{E}(K)) \cong  \mathbf{Z}^{h_{\Lambda}+1}$.
\end{corollary}
%************************************************************************************      
\begin{proof}
Indeed, the rank of the abelian group $K_0(\mathscr{A}_{RM}\rtimes \mathscr{E}(K))$ is equal to the number 
of generators of the pseudo-lattice $\tau(K_0(\mathscr{A}_{RM}\rtimes \mathscr{E}(K)))\subset\mathbf{R}$.
 It follows from  corollary \ref{cor3.6},  that such a number is equal to $h_{\Lambda}+1$.  Corollary \ref{cor3.7}
 follows. 
\end{proof}

\bigskip
Theorem \ref{thm1.1} follows from the corollaries \ref{cor3.6} and \ref{cor3.7}.

%**************************************************************************
\subsection{Proof of corollary \ref{cor1.2}}
%***************************************************************************
%**************************************************************************
\subsubsection{Part I: ~Formula $rk~\mathscr{E}(K) = h_{\Lambda}-1$}
%***************************************************************************
Let $\mathscr{E}(K)$ be an elliptic curve over the number field $K$.  
The Mordell-Weil Theorem says that $\mathscr{E}(K)\cong \mathbf{Z}^r\oplus \mathscr{E}_{tors}(K)$,
where $r=rk~\mathscr{E}(K)$ and $\mathscr{E}_{tors}(K)$ is a finite abelian group.

Consider again the pseudo-lattice   $\tau(K_0(\mathscr{A}_{RM}\rtimes \mathscr{E}(K)))\subset\mathbf{R}$
and substitute  $\mathscr{E}(K)\cong \mathbf{Z}^r\oplus \mathscr{E}_{tors}(K)$: 
%*******************************************************************
\begin{equation}\label{eq3.8}
%\left\{
\begin{array}{lll}
\tau\left[K_0(\mathscr{A}_{RM}\rtimes \mathscr{E}(K))\right] = \tau\left[K_0(\mathscr{A}_{RM}\rtimes (\mathbf{Z}^r\oplus \mathscr{E}_{tors}(K)))\right]=&&\\
&&\\
= \tau\left[K_0(\mathscr{A}_{RM}\rtimes \mathbf{Z}^r) \oplus  K_0(\mathscr{A}_{RM}\rtimes \mathscr{E}_{tors}(K))\right] =&&\\
&&\\
=\tau\left[K_0(\mathscr{A}_{RM}\rtimes \mathbf{Z}^r)\right] +  \tau\left[K_0(\mathscr{A}_{RM}\rtimes \mathscr{E}_{tors}(K))\right].&&
\end{array}
%\right.
\end{equation}
%*****************************************************************

In the last line of (\ref{eq3.8}) we have the following two terms:

\medskip
(i)      $\tau\left[K_0(\mathscr{A}_{RM}\rtimes \mathscr{E}_{tors}(K))\right]={1\over k}(\mathbf{Z}+\theta\mathbf{Z})$,
where $k\ge 2$ is an integer depending on the order of finite group $\mathscr{E}_{tors}$; we refer the reader to
[Echterhoff,  L\"uck,  Phillips \&  Walters   2010]  \cite[Theorem 0.1]{EchLuPhiWa1} for the proof of this fact.

\smallskip
(ii)  $\tau\left[K_0(\mathscr{A}_{RM}\rtimes \mathbf{Z}^r)\right]=\mathbf{Z}+\theta\mathbf{Z}+\lambda_1\mathbf{Z}+\dots+
\lambda_r\mathbf{Z}$. For $r=0$ this formula follows from [Echterhoff,  L\"uck,  Phillips \&  Walters   2010]  \cite[Theorem 0.1]{EchLuPhiWa1}
after  one rescales pseudo-lattice ${1\over k}(\mathbf{Z}+\theta\mathbf{Z})\subset\mathbf{R}$ to the pseudo-lattice 
$\mathbf{Z}+\theta\mathbf{Z}\subset\mathbf{R}$,   see remark \ref{rmk3.2}. 
For $r=1$ the formula follows from [Farsi \& Watling 1994] \cite[Proposition 19]{FaWa1}. 
For $r\ge 2$ the formula is proved by an induction. Namely, it is verified directly that the case  $i+1$ adds
an extra generator $\lambda_{i+1}$  of the pseudo-lattice $\mathbf{Z}+\theta\mathbf{Z}+\lambda_1\mathbf{Z}+\dots+
\lambda_i\mathbf{Z}$ corresponding to the case $i$.

\bigskip
It follows from (i) and (ii) that  after a scaling, one gets the following inclusion
of the pseudo-lattices:
%*********************************************************************************
\begin{equation}\label{eq3.9}
\tau\left[K_0(\mathscr{A}_{RM}\rtimes \mathscr{E}_{tors}(K))\right]
\subseteq
\tau\left[K_0(\mathscr{A}_{RM}\rtimes \mathbf{Z}^r)\right]. 
\end{equation}
%*************************************************************************************
From   (\ref{eq3.8}) and (\ref{eq3.9})  we get the following equality:
%*********************************************************************************
\begin{equation}\label{eq3.10}
\tau\left[K_0(\mathscr{A}_{RM}\rtimes \mathscr{E}(K))\right]= 
\tau\left[K_0(\mathscr{A}_{RM}\rtimes \mathbf{Z}^r)\right]. 
\end{equation}
%*************************************************************************************
Using formula (\ref{eq1.2}) and calculations of item (ii),  one obtains from (\ref{eq3.10}) 
the following equation:
%*********************************************************************************
\begin{equation}\label{eq3.11}
\mathbf{Z}+\theta\mathbf{Z}+\lambda_1\mathbf{Z}+\dots+
\lambda_{h_{\Lambda}-1}\mathbf{Z}=
\mathbf{Z}+\theta\mathbf{Z}+\lambda_1\mathbf{Z}+\dots+
\lambda_r\mathbf{Z}. 
\end{equation}
%*************************************************************************************
 It is easy to see,  that equation (\ref{eq3.11})  is solvable if and only if 
 $r=h_{\Lambda}-1$.  In other words, $r:= rk ~\mathscr{E}(K)=  h_{\Lambda}-1$.
 Part I of Corollary \ref{cor1.2} follows.

%**************************************************************************
\subsubsection{Part II: ~Formula $\Sha (\mathscr{E}(K)) \cong  ~Cl~(\Lambda)\oplus Cl~(\Lambda)$}
%***************************************************************************
Let us make general remarks and outline the main ideas of the proof. 

A relation between quadratic number fields and ranks of elliptic curves
has been known for a while  [Goldfeld 1976] \cite{Gol1}. 
In fact, the  famous Birch and Swinnerton-Dyer  Conjecture uses the relation
to compare  (special values of) the Dirichlet $L$-functions of a number field 
with the Hasse-Weil $L$-function of an elliptic curve [Swinnerton-Dyer 1967]  \cite{Swi1}. 
Let us mention a recent generalization  of this idea by [Bloch \& Kato 1990]  \cite{BloKa1}.   

Our idea is to show that there exists a natural correspondence between 
 the  arithmetic  of ideals of  the real quadratic fields and the Hasse principle 
for elliptic curves.  Namely, denote by $\mathscr{A}_{RM}^{(1)},\dots, \mathscr{A}_{RM}^{(h_{\Lambda})}$
the companion noncommutative tori corresponding to the $\mathscr{A}_{RM}$ \cite[p.163]{N};
simply speaking, these are  pairwise non-isomorphic algebras $\mathscr{A}_{RM}^{(i)}$,
 such that $End~K_0(\mathscr{A}_{RM}^{(i)})\cong\Lambda$ 
for all $1\le i\le h_{\Lambda}$. Since the companion algebras $\mathscr{A}_{RM}^{(i)}$
have the same endomorphisms, so will be their quantum dynamics,  i.e. the 
crossed product  $\mathscr{A}_{RM}^{(i)}\rtimes \mathscr{E}(K)$, see lemma \ref{lm3.8}.

On the other hand, we establish a natural isomorphism between the abelian groups 
$K_0(\mathscr{A}_{RM})\cong  H^1(K,\mathscr{E})$ and 
$K_0(\mathscr{A}_{RM}\rtimes \mathscr{E}(K))  \cong \prod_v H^1(K_v,\mathscr{E})$,
see lemma \ref{lm3.9}. In view of formula (\ref{eq2.3}), this means that the preimage of each 
cocycle in  $\prod_v H^1(K_v,\mathscr{E})$ under the homomorphism $\omega$ 
consists of the $h_{\Lambda}\ge 1$ distinct cocycles of the  $H^1(K,\mathscr{E})$.
In other words,  we get an inclusion $Cl~(\Lambda)\subset\Sha (\mathscr{E}(K))$. 

A precise formula is derived from  Atiyah's  pairing between the $K$-theory and 
the $K$-homology of $C^*$-algebras, see e.g. \cite[Section 10.2]{N}.  Namely,  it is known that the 
$K^0(\mathscr{A}_{\theta})\cong K_0(\mathscr{A}_{\theta})$,  where 
$K^0(\mathscr{A}_{\theta})$ is the zero $K$-homology group of the noncommutative
torus $\mathscr{A}_{\theta}$ [Hadfield 2004]  \cite[Proposition 4]{Had1}.  Repeating the argument for the group $K^0(\mathscr{A}_{RM})$,
we get another subgroup $Cl~(\Lambda)\subset \Sha (\mathscr{E}(K))$. 
In view of the Atiyah pairing,  one gets  $\Sha (\mathscr{E}(K))\cong ~Cl~(\Lambda)\oplus Cl~(\Lambda)$. 
We pass to a detailed argument by splitting 
the proof in a series of lemmas.

\bigskip
%***********************************************************************************
\begin{lemma}\label{lm3.8}
The $\mathscr{A}_{RM}^{(i)}$ and $\mathscr{A}_{RM}^{(j)}$ are companion noncommutative tori,
if and only if,  the crossed products  $\mathscr{A}_{RM}^{(i)}\rtimes \mathscr{E}(K)\cong \mathscr{A}_{RM}^{(j)}\rtimes \mathscr{E}(K)$
are Morita equivalent $C^*$-algebras.
\end{lemma}
%************************************************************************************      
\begin{proof}
(i)  Let $\mathscr{A}_{RM}^{(i)}$ and $\mathscr{A}_{RM}^{(j)}$ be companion noncommutative tori. 
In this case we have: 
%*******************************************************************************
\begin{equation}\label{eq3.12}
End~K_0(\mathscr{A}_{RM}^{(i)})\cong End~K_0(\mathscr{A}_{RM}^{(j)})\cong\Lambda.
\end{equation}
%************************************************************************************
Denote by 
%*******************************************************************************
\begin{equation}\label{eq3.13}
\widetilde{End}~(\mathscr{A}_{RM}^{(i)})\cong \widetilde{End}~(\mathscr{A}_{RM}^{(j)})
\end{equation}
%************************************************************************************
a pull back of the isomorphism (\ref{eq3.12}) to the category of noncommutative tori.

Recall that the crossed product  $\mathscr{A}_{RM}^{(i)}\rtimes \mathscr{E}(K)$ is an extension
of the algebra $\mathscr{A}_{RM}^{(i)}$ by  the elements $v\in \mathscr{A}_{RM}^{(i)}\rtimes \mathscr{E}(K)$,  
such that each  $\phi\in \widetilde{End}~(\mathscr{A}_{RM}^{(i)})$ becomes an inner endomorphism,
i.e. $\phi(u)=v^{-1}uv$ for  every $u\in \mathscr{A}_{RM}^{(i)}$, see Section 2.1. 
From (\ref{eq3.13}) we conclude that the algebras $\mathscr{A}_{RM}^{(i)}$ and $\mathscr{A}_{RM}^{(j)}$
must have $\ast$-isomorphic extensions. In other words, the corresponding crossed products 
$\mathscr{A}_{RM}^{(i)}\rtimes \mathscr{E}(K)$ and $\mathscr{A}_{RM}^{(j)}\rtimes \mathscr{E}(K)$
are isomorphic up to an adjustment of generators of the extension, i.e. the Morita equivalence. 
The `only if' part of lemma \ref{lm3.8} is proved. 

\bigskip
(ii)  Let $\mathscr{A}_{RM}^{(i)}\rtimes \mathscr{E}(K)\cong \mathscr{A}_{RM}^{(j)}\rtimes \mathscr{E}(K)$
be the Morita equivalent crossed products.  Let us show that $\mathscr{A}_{RM}^{(i)}$ and $\mathscr{A}_{RM}^{(j)}$ are
 companion noncommutative tori.  Indeed, our assumption implies immediately an isomorphism 
$\widetilde{End}~(\mathscr{A}_{RM}^{(i)})\cong \widetilde{End}~(\mathscr{A}_{RM}^{(j)})$. 
We apply the $K_0$-functor and we get an isomorphism   
$End~K_0(\mathscr{A}_{RM}^{(i)})\cong End~K_0(\mathscr{A}_{RM}^{(j)})\cong\Lambda$.
In other words,  the $\mathscr{A}_{RM}^{(i)}$ and $\mathscr{A}_{RM}^{(j)}$ are
 companion algebras. The `if' part of lemma \ref{lm3.8} is proved.
\end{proof}

%***********************************************************************************
\begin{lemma}\label{lm3.9}
Let  $ H^1(K,\mathscr{E})$ and  $H^1(K_v,\mathscr{E})$ be the first Galois 
cohomology over the field $K$ and over the completion $K_v$ of $K$,
respectively. There exists a natural isomorphism between the following 
groups: 
%*******************************************************************
\begin{equation}\label{eq3.14}
\left\{
\begin{array}{ccc}
H^1(K,\mathscr{E}) &\cong & K_0(\mathscr{A}_{RM}),\\
&&\\
 \prod_v H^1(K_v,\mathscr{E}) & \cong & K_0(\mathscr{A}_{RM}\rtimes \mathscr{E}(K)).
\end{array}
\right.
\end{equation}
%*****************************************************************
\end{lemma}
%************************************************************************************      
\begin{proof}
(i) Let us show  that  $H^1(K,\mathscr{E}) \cong  K_0(\mathscr{A}_{RM})$.  
Indeed, such an isomorphism is a special case of \cite[Theorem 1.1]{Nik1}
saying that $H^1(Gal(\mathbf{C}|K), Aut_{\mathbf{C}}^{~ab}(V))\cong (K_0(\mathscr{A}_V), K_0^+(\mathscr{A}_V))$. 
For that, one has to restrict to the case $V=\mathscr{E}(K)$ 
and notice that $\mathscr{A}_V=\mathscr{A}_{RM}$.  On the other hand,
since  $\mathscr{E}(K)$ is an algebraic group, one gets 
$Aut^{~ab}_{\mathbf{C}}(\mathscr{E}(K))\cong \mathscr{E}(K)$.   
The rest of the formula follows from the definition of the group 
$H^1(K,\mathscr{E})$.

\bigskip
(ii)  Let us prove that  $\prod_v H^1(K_v,\mathscr{E})  \cong  K_0(\mathscr{A}_{RM}\rtimes \mathscr{E}(K))$.
An idea of the proof is to construct an AF-algebra, $\mathbb{A}$, connected to the profinite group  
 $\prod_v H^1(K_v,\mathscr{E})$;  we refer the reader to Section 2.1.4  or  [Blackadar 1986]  
 \cite[Chapter 7]{B} for the definition of an AF-algebra.
  Next we show that the crossed product  $\mathscr{A}_{RM}\rtimes \mathscr{E}(K)$
 embeds into $\mathbb{A}$, so that  $K_0(\mathscr{A}_{RM}\rtimes \mathscr{E}(K))\cong K_0(\mathbb{A})$. 
 The rest of the proof will follow from the  properties of the AF-algebra $\mathbb{A}$.  
 We pass to a detailed argument.
 
 \smallskip
 Recall that  $\prod_v H^1(K_v,\mathscr{E})$ is a profinite group, i.e.
 %*******************************************************************
\begin{equation}\label{eq3.15}
\prod_v H^1(K_v,\mathscr{E})\cong \varprojlim G_k, 
\end{equation}
%*****************************************************************
where $G_k= \prod_{i=1}^k H^1(K_{v_i},\mathscr{E})$ is a finite group. 
 Consider a group algebra  
%****************************************************************************
\begin{equation}\label{eq3.16}
\mathbf{C}[G_k]\cong M_{n_1}^{(k)}(\mathbf{C})\oplus\dots\oplus M_{n_h}^{(k)}(\mathbf{C})
\end{equation}
%****************************************************************************
corresponding to $G_k$.  Notice that the $\mathbf{C}[G_k]$ is  a finite-dimensional 
$C^*$-algebra. The inverse limit (\ref{eq3.15})  defines an 
ascending sequence of the finite-dimensional $C^*$-algebras:
%****************************************************************************
\begin{equation}\label{eq3.17}
\mathbb{A}:=
\varprojlim M_{n_1}^{(k)}(\mathbf{C})\oplus\dots\oplus 
M_{n_h}^{(k)}(\mathbf{C}).  
\end{equation}
%****************************************************************************
In other words, the  limit $\mathbb{A}$ is an AF-algebra, such that 
$K_0(\mathbb{A})\cong \prod_v H^1(K_v,\mathscr{E})$.

\smallskip 
To prove that  $K_0(\mathscr{A}_{RM}\rtimes \mathscr{E}(K))\cong K_0(\mathbb{A})$,
we shall use the ``rigidity principle'' described  in Section 3.1. 
Namely, the extension  $H^1(K,\mathscr{E})\subset \prod_v H^1(K_v,\mathscr{E})$
is defined solely by the group $H^1(K,\mathscr{E})$  [Silverman 1985] \cite[Appendix B]{S}.
 Since  $H^1(K,\mathscr{E})\cong  K_0(\mathscr{A}_{RM})$ and 
 $\prod_v H^1(K_v,\mathscr{E})\cong K_0(\mathbb{A})$, 
 we conclude that the extension $K_0(\mathscr{A}_{RM})\subset  K_0(\mathbb{A})$
is defined by the group $K_0(\mathscr{A}_{RM})$ alone. 
But the extension $K_0(\mathscr{A}_{RM})\subset K_0(\mathscr{A}_{RM}\rtimes \mathscr{E}(K))$
is the only extension with such a property. Thus 
$K_0(\mathbb{A})\cong K_0(\mathscr{A}_{RM}\rtimes \mathscr{E}(K))$ and the crossed product
$\mathscr{A}_{RM}\rtimes \mathscr{E}(K)$ embeds into the AF-algebra $\mathbb{A}$. 

To finish the proof of lemma \ref{lm3.9}, we recall that  $K_0(\mathbb{A})\cong \prod_v H^1(K_v,\mathscr{E})$
and therefore $\prod_v H^1(K_v,\mathscr{E})  \cong  K_0(\mathscr{A}_{RM}\rtimes \mathscr{E}(K))$.
\end{proof}

%***********************************************************************************
\begin{lemma}\label{lm3.10}
$\Sha (\mathscr{E}(K))\cong ~Cl~(\Lambda)\oplus Cl~(\Lambda)$.
\end{lemma}
%************************************************************************************      
\begin{proof}
Let $\left\{\mathscr{A}_{RM}^{(i)}\right\}_{i=1}^{h_{\Lambda}}$ be 
the companion noncommutative tori of  the $\mathscr{A}_{RM}$. 
Consider a group homomorphism 
%****************************************************************************
\begin{equation}\label{eq3.18}
h:  K_0(\mathscr{A}_{RM})\to K_0(\mathscr{A}_{RM}\rtimes \mathscr{E}(K)),
\end{equation}
%****************************************************************************
induced by the the natural embedding 
$\mathscr{A}_{RM}\hookrightarrow \mathscr{A}_{RM}\rtimes \mathscr{E}(K)$. 
It follows from lemma \ref{lm3.8}, that 
%****************************************************************************
\begin{equation}\label{eq3.19}
h(K_0(\mathscr{A}_{RM}^{(i)}))=\mathbf{Z}+\theta\mathbf{Z} 
\quad\hbox{for all} \quad 1\le i\le h_{\Lambda}. 
\end{equation}
%**************************************************************************** 
In other words, one gets  $Ker~h\cong Cl~(\Lambda)$,
where $Cl~(\Lambda)$ is the class group of the order $\Lambda$ in the real 
quadratic field $\mathbf{Q}(\theta)$. 

\smallskip
But $K_0(\mathscr{A}_{RM})\cong H^1(K,\mathscr{E})$ and 
$K_0(\mathscr{A}_{RM}\rtimes \mathscr{E}(K))\cong  \prod_v H^1(K_v,\mathscr{E})$,
see lemma \ref{lm3.9}. Therefore, in view of the formulas (\ref{eq2.3}) and (\ref{eq2.4}),
the abelian group $Cl~(\Lambda)$ is an obstacle to the Hasse principle for the elliptic 
curve $\mathscr{E}(K)$. In other words, $Cl~(\Lambda)\subset \Sha (\mathscr{E}(K))$.

\medskip
To calculate an exact relation between  the groups $Cl~(\Lambda)$ and  $\Sha (\mathscr{E}(K))$,
recall that the $K$-homology is the dual theory to the $K$-theory, see e.g.  [Blackadar 1986] \cite[Section 16.3]{B}. 
Roughly speaking, cocycles in $K$-theory are represented by vector bundles.  Atiyah proposed using elliptic operators to
 represent the $K$-homology cycles.  An elliptic operator can be twisted by a vector bundle, and the Fredholm index of the 
 twisted operator defines a  pairing between the $K$-homology and the $K$-theory with values in $\mathbf{Z}$.

In partucular, it is known that for the algebra  $\mathscr{A}_{\theta}$,  it holds
$K^0(\mathscr{A}_{\theta})\cong K_0(\mathscr{A}_{\theta})$,  where 
$K^0(\mathscr{A}_{\theta})$ is the zero $K$-homology group of 
$\mathscr{A}_{\theta}$ [Hadfield 2004]  \cite[Proposition 4]{Had1}.  
Repeating the argument for the group $K^0(\mathscr{A}_{RM})$,
one can prove an analog of theorem \ref{thm1.1} for such a group. 
In other words,  we get another subgroup $Cl~(\Lambda)\subset \Sha (\mathscr{E}(K))$. 
Since there are no other duals to the $K$-theory of $C^*$-algebras, we conclude from 
the Atiyah pairing,  that   $\Sha (\mathscr{E}(K))\cong ~Cl~(\Lambda)\oplus Cl~(\Lambda)$. 
Lemma \ref{lm3.10} is proved. 
\end{proof}

\bigskip
Part II of Corollary \ref{cor1.2} follows from lemma \ref{lm3.10}.

\bigskip
%************************************************************************************************
\begin{remark}\label{rmk3.11}%(Long remark about CFT and Hasse principle.)
Construction of  generators  of 
$\mathscr{E}(K)$ is similar to such of  the principal ideals of  the field $k$.
Namely, it is  known that  there exist non-principal  ideals  in  $\Lambda\subseteq O_k$;   an obstruction  is  a non-trivial 
group $Cl~(\Lambda)$.   However,  this can be done in a bigger field $\mathcal{K}=\mathcal{K}_{ab}$;
 there exists a  finite  extension $k\subseteq \mathcal{K}$, such that  every ideal of $\Lambda$ is principal in the ring $O_{\mathcal{K}}$.   
 Likewise,  one cannot construct a generator of $\mathscr{E}(K)$ by a  finite  descent in general;  an obstruction is a non-trivial 
 group $\Sha (\mathscr{E}(K))$.  However, in an  extension  $\mathscr{A}_{RM}\rtimes \mathscr{E}(K)$ of the 
 coordinate ring $\mathscr{A}_{RM}$ of $\mathscr{E}(K)$, the descent will be always finite and  give a generator of the $\mathscr{E}(K)$.
Such an analogy explains the formula  $\Sha (\mathscr{E}(K))\cong ~Cl~(\Lambda)\oplus Cl~(\Lambda)$ on 
an intuitive level.  Notice also that the   $\mathscr{A}_{RM}\rtimes \mathscr{E}(K)$  is the coordinate ring of an abelian variety $A(K)$
which is  related to  the Euler variety $V_E$ coming from the 
continued fraction of $\theta$ \cite[Section 6.2.1]{N}. 
 \end{remark}
%************************************************************************************************

  %%%%%%%%%%%%%%
\section*{Data availability}
%%%%%%%%%%%%%%%
  
 Data sharing not applicable to this article as no datasets were generated or analyzed during the current study.
   
%%%%%%%%%%%%%%
\section*{Conflict of interest}
%%%%%%%%%%%%%%%
On behalf of all co-authors, the corresponding author states that there is no conflict of interest.
  
  %%%%%%%%%%%%%%

\bibliographystyle{amsplain}

%**********************************************************

\end{document}